\documentclass[12pt]{amsart}
\usepackage{fourier}
\usepackage[T1]{fontenc}

\usepackage{graphics}
\usepackage{amsfonts,amssymb,color}
\usepackage[mathscr]{eucal}
\usepackage{amsmath, amsthm}
\usepackage{mathrsfs}
\usepackage{amsbsy}
\usepackage{dsfont}
\usepackage{bbm}
\usepackage{wasysym}
\usepackage{stmaryrd}
\usepackage{url}

\input xypic
\xyoption{all}

\makeindex
\makeglossary

\begin{document}
\baselineskip = 16pt

\newcommand \ZZ {{\mathbb Z}}
\newcommand \NN {{\mathbb N}}
\newcommand \RR {{\mathbb R}}
\newcommand \PR {{\mathbb P}}
\newcommand \AF {{\mathbb A}}
\newcommand \GG {{\mathbb G}}
\newcommand \QQ {{\mathbb Q}}
\newcommand \CC{{\mathbb C}}
\newcommand \bcA {{\mathscr A}}
\newcommand \bcC {{\mathscr C}}
\newcommand \bcD {{\mathscr D}}
\newcommand \bcF {{\mathscr F}}
\newcommand \bcG {{\mathscr G}}
\newcommand \bcH {{\mathscr H}}
\newcommand \bcM {{\mathscr M}}
\newcommand \bcJ {{\mathscr J}}
\newcommand \bcL {{\mathscr L}}
\newcommand \bcO {{\mathscr O}}
\newcommand \bcP {{\mathscr P}}
\newcommand \bcQ {{\mathscr Q}}
\newcommand \bcR {{\mathscr R}}
\newcommand \bcS {{\mathscr S}}
\newcommand \bcV {{\mathscr V}}
\newcommand \bcW {{\mathscr W}}
\newcommand \bcX {{\mathscr X}}
\newcommand \bcY {{\mathscr Y}}
\newcommand \bcZ {{\mathscr Z}}
\newcommand \goa {{\mathfrak a}}
\newcommand \gob {{\mathfrak b}}
\newcommand \goc {{\mathfrak c}}
\newcommand \gom {{\mathfrak m}}
\newcommand \gon {{\mathfrak n}}
\newcommand \gop {{\mathfrak p}}
\newcommand \goq {{\mathfrak q}}
\newcommand \goQ {{\mathfrak Q}}
\newcommand \goP {{\mathfrak P}}
\newcommand \goM {{\mathfrak M}}
\newcommand \goN {{\mathfrak N}}
\newcommand \uno {{\mathbbm 1}}
\newcommand \Le {{\mathbbm L}}
\newcommand \Spec {{\rm {Spec}}}
\newcommand \Gr {{\rm {Gr}}}
\newcommand \Pic {{\rm {Pic}}}
\newcommand \Jac {{{J}}}
\newcommand \Alb {{\rm {Alb}}}
\newcommand \Corr {{Corr}}
\newcommand \Chow {{\mathscr C}}
\newcommand \Sym {{\rm {Sym}}}
\newcommand \alb {{\rm {alb}}}
\newcommand \Prym {{\rm {Prym}}}
\newcommand \cha {{\rm {char}}}
\newcommand \eff {{\rm {eff}}}
\newcommand \tr {{\rm {tr}}}
\newcommand \Tr {{\rm {Tr}}}
\newcommand \pr {{\rm {pr}}}
\newcommand \ev {{\it {ev}}}
\newcommand \cl {{\rm {cl}}}
\newcommand \interior {{\rm {Int}}}
\newcommand \sep {{\rm {sep}}}
\newcommand \td {{\rm {tdeg}}}
\newcommand \alg {{\rm {alg}}}
\newcommand \im {{\rm im}}
\newcommand \gr {{\rm {gr}}}
\newcommand \op {{\rm op}}
\newcommand \Hom {{\rm Hom}}
\newcommand \Hilb {{\rm Hilb}}
\newcommand \Sch {{\mathscr S\! }{\it ch}}
\newcommand \cHilb {{\mathscr H\! }{\it ilb}}
\newcommand \cHom {{\mathscr H\! }{\it om}}
\newcommand \colim {{{\rm colim}\, }} 
\newcommand \End {{\rm {End}}}
\newcommand \coker {{\rm {coker}}}
\newcommand \id {{\rm {id}}}
\newcommand \van {{\rm {van}}}
\newcommand \spc {{\rm {sp}}}
\newcommand \Ob {{\rm Ob}}
\newcommand \Aut {{\rm Aut}}
\newcommand \cor {{\rm {cor}}}
\newcommand \Cor {{\it {Corr}}}
\newcommand \res {{\rm {res}}}
\newcommand \red {{\rm{red}}}
\newcommand \Gal {{\rm {Gal}}}
\newcommand \PGL {{\rm {PGL}}}
\newcommand \Bl {{\rm {Bl}}}
\newcommand \Sing {{\rm {Sing}}}
\newcommand \spn {{\rm {span}}}
\newcommand \Nm {{\rm {Nm}}}
\newcommand \inv {{\rm {inv}}}
\newcommand \codim {{\rm {codim}}}
\newcommand \Div{{\rm{Div}}}
\newcommand \sg {{\Sigma }}
\newcommand \DM {{\sf DM}}
\newcommand \Gm {{{\mathbb G}_{\rm m}}}
\newcommand \tame {\rm {tame }}
\newcommand \znak {{\natural }}
\newcommand \lra {\longrightarrow}
\newcommand \hra {\hookrightarrow}
\newcommand \rra {\rightrightarrows}
\newcommand \ord {{\rm {ord}}}
\newcommand \Rat {{\mathscr Rat}}
\newcommand \rd {{\rm {red}}}
\newcommand \bSpec {{\bf {Spec}}}
\newcommand \Proj {{\rm {Proj}}}
\newcommand \pdiv {{\rm {div}}}
\newcommand \CH {{\it {CH}}}
\newcommand \wt {\widetilde }
\newcommand \ac {\acute }
\newcommand \ch {\check }
\newcommand \ol {\overline }
\newcommand \Th {\Theta}
\newcommand \cAb {{\mathscr A\! }{\it b}}

\newenvironment{pf}{\par\noindent{\em Proof}.}{\hfill\framebox(6,6)
\par\medskip}

\newtheorem{theorem}[subsection]{Theorem}
\newtheorem{conjecture}[subsection]{Conjecture}
\newtheorem{proposition}[subsection]{Proposition}
\newtheorem{lemma}[subsection]{Lemma}
\newtheorem{remark}[subsection]{Remark}
\newtheorem{remarks}[subsection]{Remarks}
\newtheorem{definition}[subsection]{Definition}
\newtheorem{corollary}[subsection]{Corollary}
\newtheorem{example}[subsection]{Example}
\newtheorem{examples}[subsection]{examples}

\title{Involutions on algebraic surfaces and the Generalised Bloch's conjecture}
\author{Kalyan Banerjee}

\address{HRI, India}

\email{banerjeekalyan@hri.res.in}

\begin{abstract}
In this note we are going to consider a smooth projective surface equipped with an involution and study the action of the involution at the level of Chow group of zero cycles.
\end{abstract}

\maketitle

\section{Introduction}
In this note we want to consider the generalised Bloch conjecture \cite{Vo}  [conjecture 11.19], which says that the action of a degree two correspondence on the Chow group of zero cycles on a smooth projective surface is determined by its cohomology class in $H^4(S\times S,\ZZ)$. This is equivalent to the following: let $\Gamma$ be a correspondence of codimension $2$ on $S\times T$ where $S,T$ are smooth projective surfaces over the field of complex numbers. Suppose that $\Gamma^*$ vanishes on $H^0(T,\Omega^2_T)$ then the homomorphism $\Gamma_*$ from $\CH_0(S)$ to $\CH_0(T)$ vanishes on the kernel of the albanese map $alb_S:\CH_0(S)\to Alb(S)$.

In  \cite{Voi}, the conjecture was proved for a symplectic involution on a $K3$ surface. In this paper the author consider an automorphism of order two $i$ of the given K3 surface, such that $i^*$ acts as identity on globally holomorphic $2$-forms, then $i_*$ acts as identity on $\CH_0$ of the K3 surface. Also the similar question was considered in \cite{G} for intersection of quadrics and cubics in $\PR^4$ which are examples of K3 surfaces. Also in
\cite{HK} the question was considered and proved for certain examples of K3 surfaces equipped with a symplectomorphism.

In this note we prove the following theorem:
\smallskip

\textit{Let $S$ be a smooth surface  admitting a $2:1$ map $f$ to a surface $F$  admitting an elliptic pencil such that the corresponding Jacobian fibration $J$ admits an ample line bundle $L$ such that the genus of the curves in $|L|$ are bounded by some positive integer $n_0$. Let $i$ be the involution on $S$ arising from the $2:1$ map. Then the group of invariants of $A_0(S)$, given by
$$\{z\in A_0(S):i(z)=z\}$$
is finite dimensional.}

\smallskip

Our method of proof goes in the line of the proof of Bloch's conjecture for surfaces not of general type with $p_g=0$ as in \cite{BKL} and of the arguments present in \cite{Voi}.

{\small \textbf{Acknowledgements:} The author would like to thank the ISF-UGC project for funding this project and also thanks the hospitality of Indian Statistical Institute, Bangalore Center for hosting this project. The author is indebted to Ramesh Sreekantan for suggesting this problem to the author and for many helpful discussions on the theme of the paper. Lastly the author is grateful to Chuck Weibel for constructive criticism on improving the exposition of the paper and for his advice to improve \ref{theorem2}.

We assume that the ground field is algebraically closed and of characteristic zero.}


\section{The Bloch-Kas-Liebarman technique}

In this section we prove the following theorem:

\begin{theorem}
\label{theorem2}
Let $S$ be a smooth surface  admitting a $2:1$ map $f$ to a surface $F$  admitting an elliptic pencil such that the corresponding Jacobian fibration $J$ admits an ample line bundle $L$ such that the genus of the curves in $|L|$ are bounded by some positive integer $n_0$ . Let $i$ be the involution on $S$ arising from the $2:1$ map. Then the group of invariants of $A_0(S)$, given by
$$\{z\in A_0(S):i(z)=z\}$$
is finite dimensional.
\end{theorem}
\begin{proof}
To prove that the group of $i$-invariants of  the Chow group of degree zero cycles of $S$, we follow the Bloch-Kas-Lieberman technique as presented in \cite{BKL}. First consider the pencil of elliptic curves on the surface $F$. That is a  map from $F\dashrightarrow L$, where $L$ is isomorphic to $\PR^1$.

Suppose that a pencil of curves on a surface $F$ can be given by choosing a projective line $L$ in $\PR(H^0(F,D)):=|D|$, where $D$ is a line bundle on $F$. So every element $t$ of this projective line gives rise to a global section $\sigma_t$ of $D$, which is non-zero and well-defined upto scalar multiplication. Let $F_t$ be the curve in $F$ defined by the zero locus of $\sigma_t$. Now let $\sigma_0,\sigma_{\infty}$ be two linearly independent global sections spanning the two dimensional vector subspace of $H^0(F,D)$, underlying the line $L$. Then any element $\sigma_t$ in this vector space\ look like $\sigma_0+t\sigma_{\infty}$. Now the rational map $F\to L$ is defined by
$$x\mapsto [\sigma_0(x):\sigma_{\infty}(x)]$$
and it is not defined along the common zero locus of $\sigma_0=\sigma_{\infty}=0$.
Consider the surface
$$\wt{F}=\{(x,t)\in F\times L|x\in F_t\}$$
this is nothing but the blow-up of $F$ along the base locus of the above rational map. Then sending $(x,t)$ to $t$ defines a regular map from $\wt{F}$ to $L$. That is we blow up the base locus of the rational map $F\dashrightarrow L$. Now consider the pull-back of $\wt{F}\to F$ to $S$, call it $\wt{S}$. Then $\wt{S}$ is nothing but the blow up of $S$ along the base locus of the rational map $S\dashrightarrow L$. Observe that fixing a point $0$ in on $F$, which is in the base locus of the pencil, we have  a section of $\wt{F}\to L$ given by $t\mapsto (0,t)$. Let us continue to denote the map from $\wt{S}$ to $\wt{F}$ by $f$.

Consider the Jacobian fibration $J\to L$ corresponding to $\wt{F}\to L$. Now fix a smooth hyperplane section $Y$ of $\wt{F}$ under the embedding of $\wt{F}$ in some $\PR^N$. Let $\pi$ be the morphism from $\wt{S}\to L$ and $\pi'$ is from $\wt{F}\to L$. Let us have
$$Y\cap \pi'^{-1}(t)=\sum_{i=1}^n p_i(t)$$
Then we have a  map $g$ from $\wt{S}$ to $J$
$$q\mapsto \alb_{t}(nf(q)-\sum_i p_i(\pi(q)))$$
where $\alb_t$ is the Albanese map from $F_t$ to $J_t$, $t=\pi(q)$. It is defined because the pencil $\wt{F}\to L$ has a section. This  map is dominant as it is dominant on fibers.

Now we recall the notion of finite dimensionality in the sense of Roitman \cite{R1}: Consider a correspondence $Z$ on $S\times S$, then the image of $Z_*$ from $A_0(S)$ to $A_0(S)$ is said to be finite dimensional if there exists a smooth projective variety $W$ and a correspondence $\Gamma$ on $W\times S$ such that the image of $Z_*$ is contained in the set:
$$\{\Gamma_*(w)|w\in W\}\;.$$

\begin{lemma}
Let the image of $\lambda g_*$ be finite dimensional in the above mentioned sense. Then the group of invariants under the action of $i$, in $A_0(S)$ is finite dimensional.
\end{lemma}
\begin{proof}
The proof of this lemma follows by arguing as in \cite{BKL}[proposition 4]. To prove the claim we have to  understand the quasi-inverse of  $g$ given by a correspondence on $\wt{S}\times J$. Let $\alpha$ belong to $J$ that lies over $t\in L$. View $\alpha$ as a zero cycle on $J_t$, that is it is an element in $\Pic^0(J_t)$ (by using the section for the Jacobian fibration). Since $\Pic^0(J_t)$ is isomorphic to $J_t$, there is a unique point in $q_{i}(t)$ on $F_t$ such that $q_i(t)-p_i(t)$ is rationally equivalent to  $\alpha$. Now $f^{-1}(q_i(t))=\{q_i'(t),q_i''(t)\}, f^{-1}(p_i(t))=\{p_i'(t),p_i''(t)\}$. So we can define $\lambda$ to be
$$\alpha\mapsto \sum_i (q_i'(t)+q_i''(t))-(p_i'(t)+p_i''(t)).$$

Let $q-p$ be a zero cycle where $q,p$ are closed points on $\wt{S}$.  Then we compute $\lambda g_*(q-p+i(q-p))$. We have by definition
$$g_*(q+iq-p-ip)=nf_{*}(q+iq-p-ip)-\sum_{i=1}^n 2 (p_i(\pi(q))-p_i(\pi(p))\;,$$
let $f_(q)=q', f(p)=p'$.
Then
$$\lambda g_*(q+iq-p-i(p))=\lambda[(2nq'-2np')-(2\sum_{i=1}^n (p_{i}(\pi(q)))-p_i(\pi(p)))]$$
which can be re-written as
$$2\sum_{i=1}^n \lambda(q'-(p_{i}(\pi(q))))+\lambda(p'-(p_{i}(\pi(p))))\;.$$
Now $\lambda(q'-(p_{i}(\pi(q))))=(q+iq-p_i'(\pi(q))-p_i''(\pi(q)))$.
Therefore $\lambda g_*(q+iq-p-i(p))$ is equal to
$$2n(q+i(q)-p-i(p))-a$$
where $a$ is a zero cycle supported on $Y'=f^{-1}(Y)$. So for general hyperplane section $Y$ of $\wt{F}$, we have $Y'$ a smooth projective curve.
Therefore by Chow moving lemma, for any zero cycle $z$ of degree zero on $\wt{S}$, we have
$$2n(z+iz)-\lambda g_*(z)\;,$$
is supported on the Jacobian of $Y'$.
Suppose that we can prove that the image of $\lambda g_*$ is finite dimensional in the sense that there exists a smooth projective variety $W$ and a correspondence $\Gamma$ on $W\times S$, such that the image of $\lambda g_*$ is contained in  the set

$$\{\Gamma_*(w):w\in W\}\;.$$
Then  we get that $2n(z+iz)$ is supported on $J(C)$ for a smooth projective curve $C$. Tensoring with $\QQ$ we get that
$z+iz$ is supported on $J(C)\otimes _{\ZZ}\QQ$. So it means that the group $A_0(\wt{S})^i$ of $i$-invariant elements in $A_0(\wt{S})$, is finite dimensional (rationally): in  the sense that there exists a smooth projective curve $C$, and a correspondence $\Gamma$ on $C\times \wt{S}$ such that $\Gamma_*$ from $J(C)\otimes_{\ZZ}\QQ$ to $A_0(\wt{S})^i\otimes _{\ZZ}\QQ$ is surjective.  Then by lemma 3.1 in \cite{GG} it follows that the group of $i$-invariant elements of $A_0(\wt{S})$ is finite dimensional. Since $\wt{S}$ is a blow-up of $S$ and finite dimensionality is a birational invariant, we have the group of $i$-invariant elements of $A_0(S)$ is finite dimensional.
\end{proof}

Now we prove the following:

\begin{lemma}
Let $\wt{S}, F$ be as above and $F$ is equipped with an elliptic pencil on it. Suppose there exists $L$ on $J$  ample and the genus of the curves in $|L|$ are bounded by some integer $n_0$. Let $\lambda, g$ be as above. Then the image of $\lambda g_*$ is finite dimensional.
\end{lemma}
\begin{proof}
Consider the symmetric power $\Sym^{m}\wt{S}$. Now given a zero cycle of the form $q-p$ on $\wt{S}$, we have each of $q,p$ belongs to a unique fiber for a general $q,p$. Let us fix a point $p_0$ which belongs to the exceptional locus of the blow up $\wt{S}\to S$. Then $p_0$ belongs to all the fibers of the fibration $\wt{S}\to L$. Then write $q-p=q-p_0-(p-p_0)$. Now consider the map from $\Sym^{m}\wt{S}$ to $A_0(\wt{S})$, given by
$$\sum_j P_j\mapsto \sum_j n((P_j-p_0)+i(P_j-p_0))\;. $$

Now by the above we have that
$$(P_j-p_0)+i(P_j-p_0)$$
belongs to  the Jacobian $J(C_t)$, such that $P_j,p_0$ belongs to $C_t$. Actually
$$(P_j-p_0)+i(P_j-p_0)$$
lands inside the $i$-invariant part of the involution on $J(C_t)$, denote it by $P_t$. Here $P_t$ has dimension $1$ as it is complementary to the $i$-anti-invariant part in $J(C_t)$. Indeed $J(C_t)$ has dimension
$$1+\deg(R)/2$$
by the Riemann-Hurwitz formula ($R$ is the ramification locus). Therefore the dimension of $i$-antiinvariant part is
$$\deg(R)/2,$$
hence the dimension of $P_t$ is $1$.

Therefore we have that the map from $\Sym^{m}\wt{S}$ to $A_0(\wt{S})$ factoring through the fibration $\prod_{i=1}^m \bcP$,
here $\bcP$ is the abelian fibration of the abelian varieties $\bcP_t$ over $L$. So the dimension of
$$\prod_{i=1}^m \bcP$$
is $2m$. Therefore the map from $\Sym^m \wt{S}$ to $\prod_{i=1}^m \bcP$ is generically finite.

Since the fibers of $\bcP\to L$ are of dimension $1$, we have that the fibers are elliptic curves. So we have
$$\prod_{i=1}^m \bcP$$
is a product of surfaces with elliptic fibrations. Note that $\bcP_t$ is isomorphic to $J_t$. Hence $\bcP\cong J$. Therefore the group $\im(\lambda g_*)$ is dominated by
$$\prod_{i=1}^m J$$
or by
$$\Sym^m  J\;,$$
where $m$ varies over natural numbers.
Now consider the following commutative diagram:

$$
  \diagram
   \Sym^m J_t\ar[dd]_-{} \ar[rr]^-{} & & \Sym^{m}J\ar[dd]^-{\theta} \\ \\
  A_0(J_t)\cong J_t \ar[rr]^-{} & & A_0(J)
  \enddiagram
  $$
Then the fibers of the right hand side vertical map contains the image of the fibers of the left-hand side vertical map, which are $\PR^m$ by the Riemann-Roch theorem (when $m\geq 2$). Therefore the fiber of the map
$$\Sym^m J\to A_0(J)\to \im(\lambda g_*)$$
contains a projective space $\PR^m$, as the map from $A_0(J)\to \im(\lambda g_*)$ is surjective. Actually the family of $\theta^{-1}(z+iz)$ such that $z$ is supported on $J_t$ contains a projective bundle over $J_t$. Now we choose an ample smooth projective curve $C$ in $|L|$. Since $K_J.L+L^2$ is bounded, by adjunction formula, we have that the genus of $C$ is bounded and less than $m$. We choose $m$ to be much bigger than the genus of $C$, then the divisor $\sum_i pr_i^{-1} C$ on $J^m$ intersects the fibers of the map
$$J^m \to \Sym^m J\to A_0(J)\to \im(\lambda g_*)\;.$$
Therefore the image of the above map is actually the image of $J^{m-1}\times C$. Consider the map
$$J^{m-1}\to J^{m-1}\times C\to \im(\lambda g_*)\;.$$
Then again by the previous process the fibers of the above map are supported on
$$J^{m-2}\times C^2\;.$$
By continuing this process, the fibers of the map
$$J^m\to \im(\lambda g_*)$$
are supported on $$J^{m-i}\times C^i\;.$$
Here $i$ is the number of iterations such that  $i$ is less than the dimension of the fiber of the map
$$\Sym^m J_t\to J_t$$
which is by Riemann-Roch theorem, equal to $m$. Therefore, when $i> g(C)$,
the image of
$J^m$ under $\lambda $
is same as $J^{m-i}\times C^i$, which is same as $J^{m-i}\times C^g$.
Therefore we have that
$$\lambda(J^m)=\lambda(J^{m-i+g})$$
 So the image of $J^m$ under $\lambda$ is same as the image of $J^l$ under $\lambda$ for $l<m$. Hence the image of $\lambda g_*$ is finite dimensional by \cite{Voi}[lemma 3.1].

\end{proof}

\end{proof}

\begin{corollary}
\label{cor1}
Let $S$ be the branched double cover of an elliptic K3 surface $F$, let $i$ be the involution on $S$ involution $i$. Then the corresponding Jacobian fibration $J$ associated to the elliptic pencil does not have an ample line bundle $L$ on it, such that it satisfies the following property:

Given any positive integer $m$ as in the above theorem, the number of iterations $i$ is greater than $(L^2+L.K_J)/2+1$.
\end{corollary}
\begin{proof}
If there exists such line bundles then it will follow from the previous theorem  \ref{theorem2} that $A_0(F)$ is finite dimensional, which is not true because geometric genus of $F$ is greater than zero.
\end{proof}

\end{document}